\documentclass[12pt]{article}

\usepackage[T1]{fontenc}
\usepackage{textcomp}
\usepackage{mathrsfs}

\usepackage{amssymb}
\usepackage{amsmath}
\usepackage{latexsym}
\usepackage{amscd}
\usepackage{amsthm}
\usepackage{bm}

\usepackage[dvipdfm]{graphicx}

\newcommand{\Q}{\Bbb Q}
\newcommand{\Z}{\Bbb Z}

\newcommand{\F}{\Bbb F}
\newcommand{\ga}{{\rm Gal}}
\newcommand{\f}[1]{{\frak #1}}

\newcommand{\ca}[1]{{\mathcal #1}}

\newcommand{\fulltoday}{\number\day\space \ifcase\month\or
    January\or February\or March\or April\or May\or June\or
    July\or August\or September\or October\or November\or December\fi
    \space\number\year}

\theoremstyle{plain}

\newtheorem{thm}{\indent\bf Theorem}
\newtheorem{lem}{\indent\bf Lemma}
\newtheorem{prop}{\indent\bf Proposition}
\newtheorem{cor}{\indent\bf Corollary}

\newtheorem*{conj1}{\indent\bf Conjecture}

\begin{document}
\title{
Reports on families of imaginary abelian fields with 
pseudo-null unramified Iwasawa modules
}

\author{
Satoshi FUJII\thanks{
Faculty of Education, Shimane University, 
1060 Nishikawatsucho, Matsue, Shimane, 690--8504, Japan. 
e-mail : {\tt fujiisatoshi@edu.shimane-u.ac.jp}, 
{\tt mophism@gmail.com}
}
}
\date{
}
\maketitle

\begin{abstract}
Let $p$ be a prime number. 
We show that, 
there exists an infinite family of imaginary abelian fields such that, 
the Iwasawa module of the maximal multiple $\Z_p$-extension is non trivial and 
pseudo-null for each field in the family. 
We also discuss on an application to non-abelian Iwasawa theory in the sense of Ozaki. 
\end{abstract}

\section{Introduction}

All algebraic extensions of the field of rational numbers $\Q$ are assumed to be 
contained in a fixed algebraic closure of $\Q$. 
Let $k/\Q$ be a finite extension. 
For a prime number $p$, 
let $\Z_p$ be the ring of $p$-adic integers. 
Let $\tilde{k}$ be the composite field of all $\Z_p$-extensions of $k$. 
Then, 
by theorem 3 of \cite{Iwasawa73} (see also theorem $13.4$ of \cite{Washington}), 
there exists a non negative integer $\delta$ such that 
$\ga(\tilde{k}/k)$ is isomorphic to $\Z_p^{r_2+1+\delta}$ as a topological group, 
here $r_2$ denotes the number of complex primes of $k$. 
We should remark that if Leopoldt's conjecture for $p$ and $k$ holds 
then $\delta=0$. 
It is known from Brumer's result \cite{Brumer} that if $k$ is contained in an abelian extension 
of an imaginary quadratic field then 
Leopoldt's conjecture for $p$ and $k$ holds, 
and hence $\tilde{k}/k$ is a $\Z_p^{r_2+1}$-extension. 

Let $L_{\tilde{k}}/\tilde{k}$ be the maximal unramified abelian pro-$p$ extension 
and put $X_{\tilde{k}}=\ga(L_{\tilde{k}}/\tilde{k})$, 
which is often called the unramified Iwasawa module of $\tilde{k}$. 
The Galois group $\ga(\tilde{k}/k)$ acts on $X_{\tilde{k}}$ via the inner automorphism, 
and then, 
the complete group ring 
$$
\Lambda
=
\Z_p[\![\ga(\tilde{k}/k)]\!]
=
\varprojlim_{
k\subseteq k^{\prime}\subseteq \tilde{k},
[k^{\prime}:k]<\infty
}
\Z_p[\ga(k^{\prime}/k)]
$$
of $\ga(\tilde{k}/k)$ with coefficients in $\Z_p$ acts on $X_{\tilde{k}}$, 
the projective limit is taken with respect to restriction maps of Galois groups. 
It is shown that $X_{\tilde{k}}$ is a finitely generated torsion $\Lambda$-module, 
see theorem $1$ of \cite{Greenberg73}. 
A finitely generated torsion $\Lambda$-module $M$ is called pseudo-null 
if the annihilator ideal of $M$ over $\Lambda$ is not contained in any height one 
prime ideals of $\Lambda$, 
and write $M\sim 0$ for pseudo-null $\Lambda$-modules $M$. 
For example, 
$M=0$ is a pseudo-null module since the annihilator ideal of $0$ is $\Lambda$. 
In \cite{Greenberg}, 
Greenberg proposed the following conjecture. 

\begin{conj1}[Greenberg's generalized conjecture, \cite{Greenberg}]\label{GGC}
For each prime number $p$ and each number field $k$, 
it holds that $X_{\tilde{k}}\sim 0$. 
\end{conj1}

In \cite{GC}, 
Greenberg originally proposed so called Greenberg's conjecture (GC for short), 
which asserts the finiteness of the unramified Iwasawa module 
of the cyclotomic $\Z_p$-extension of totally real fields. 
Thereafter, 
Greenberg proposed the above conjecture. 
For $\Z_p$-extensions, 
on the unramified Iwasawa module, 
it is well known that the finiteness is equivalent to the pseudo-nullity. 
If Leopoldt's conjecture holds for a totally real field $k$ and a prime number $p$, 
then the cyclotomic $\Z_p$-extension is the unique $\Z_p$-extension of $k$. 
Hence, 
Greenberg's generalized conjecture (GGC for short) is in fact a generalization of 
GC in this sense. 
No counterexamples of GC and GGC have been found yet.

In this article, 
we will discuss on families of prime numbers $p$ and number fields 
$k$ such that $X_{\tilde{k}}\sim 0$. 
As we will see in section 2, 
it had been shown that, 
for each prime number $p$, 
there exist infinitely many imaginary quadratic fields $k$ such that 
$X_{\tilde{k}}=0$. 
The main result of this article is as follows. 
\\

{\bf Main Theorem (Theorem \ref{main1} of Section 2).} 
{\it Let $p$ be a prime number. 
Then there exist infinitely many imaginary abelian fields $k$ such that $X_{\tilde{k}}\neq 0$  
and $X_{\tilde{k}}\sim 0$.}
\\

From theorem 1 of \cite{Ozaki2004}, 
for each prime number $p$, there exist infinitely many cyclotomic 
$\Z_p$-extensions $k_{\infty}$ of 
totally real fields $k$ such that 
the Iwasawa module of $k_{\infty}/k$ is non trivial and finite, 
namely pseudo-null. 
In fact Ozaki has obtained a much stronger conclusion about the structure of 
Iwasawa modules. 
Although we cannot mention on the structure of $X_{\tilde{k}}$ here, 
we will prove that at least there is an infinite family of imaginary abelian fields 
with non trivial pseudo-null Iwasawa modules.

\section{Examples of prime numbers and number fields}
In this section, 
first, 
we introduce some results 
(propositions \ref{proposition1}, 
\ref{thm1}, 
and \ref{thm2}, 
and corollary \ref{corofprop1}), 
which can be deduced directly from combining known results. 
Second we present our main result (theorem \ref{main1}). 
Let $h_k$ be the class number of a number field $k$. 
Let $P$ be a prime number or a finite prime of a number field. 
When $P$ divides an integer $b$ we write $P\mid b$, 
if not, 
we write $P\nmid b$. 

\begin{lem}[See \cite{Iwasawa}]\label{Iwasawa56}
Let $p$ be a prime number. 
Let $L/K$ be a finite $p$-extension of a finite extension $K/\Q$. 
Suppose that $L/K$ is ramified at only one prime of $K$ and is totally ramified. 
If $p\nmid h_K$ then $p\nmid h_L$. 
\end{lem}

Let $p$ be a prime number. 
It is known that $X_{\tilde{k}}$ is also defined to be the projective limit with respect to 
the norm maps of the $p$-Sylow subgroups of the ideal class groups 
of intermediate fields $k^{\prime}$ of $\tilde{k}/k$ 
which are finite over $k$. 
Hence if $p\nmid h_{k^{\prime}}$ for each $k^{\prime}$ then $X_{\tilde{k}}=0$. 
Assume that the prime number $p$ does not split in $k/\Q$ and $p\nmid h_k$.
Then, 
$\tilde{k}/k$ is totally ramified at the unique prime of $k$ lying above $p$. 
By lemma \ref{Iwasawa56} and $p\nmid h_k$, 
we have $p\nmid h_{k^{\prime}}$. 
Hence we have the following. 

\begin{prop}\label{proposition1}
Let $p$ be a prime number and $k/\Q$ a finite extension. 
If $p$ does not split in $k/\Q$ and if $p\nmid h_k$, 
then $X_{\tilde{k}}=0$. 
In particular, 
$X_{\tilde{k}}\sim 0$. 
\end{prop}

Let $k/\Q$ be a finite cyclic extension. 
There exist infinitely many prime numbers $p$ such that, 
$p$ does not divide the discriminant of $k$, 
$p\nmid h_k$, 
and that the Frobenius of $p$ in $k/\Q$ generates $\ga(k/\Q)$. 
For such prime numbers $p$, 
it holds that $X_{\tilde{k}}=0$ by proposition \ref{proposition1}. 
Thus we have the following.

\begin{cor}\label{corofprop1}
Let $k/\Q$ be a finite cyclic extension. 
Then there exist infinitely many prime numbers $p$ such that 
$X_{\tilde{k}}=0$. 
\end{cor}
 
In the rest of Section 2, 
we discuss on the existence of families of number fields with pseudo-null Iwasawa modules 
for each prime number. 
Minardi proved the following result. 

\begin{prop}[Minardi \cite{Minardi}]\label{Minardi}
Let $p$ be a prime number and $k$ an imaginary quadratic field. 
If $p\nmid h_k$, 
then $X_{\tilde{k}}\sim 0$. 
\end{prop}

In particular, 
if $k$ is an imaginary quadratic field of class number $1$, 
then $X_{\tilde{k}}\sim 0$ for all prime numbers $p$. 

Let $p=2$. 
By Dirichlet's theorem, 
there exist infinitely many prime numbers $q$ such that $q\equiv 7\bmod{8}$. 
Put $k=\Q(\sqrt{-q})$. 
Then, 
the prime $2$ splits in $k$, 
and by genus theory we have $2\nmid h_k$. 
Let $p=3$. 
By Nakagawa--Horie's theorem \cite{Nakagawa-Horie}, 
there exist infinitely many imaginary quadratic fields $k$ such that 
the prime $3$ splits in $k$ and that $3\nmid h_k$. 
Let $p\geq 5$. 
By Horie--\^{O}hnishi's theorem \cite{Horie-Ohnishi}, 
there exist infinitely many imaginary quadratic fields $k$ such that 
the prime $p$ splits in $k$ and that $p\nmid h_k$. 
Therefore, 
for each prime number $p$, 
there exist infinitely many of imaginary quadratic fields 
$k$ in which $p$ splits such that $p\nmid h_k$. 
We then have the following by proposition \ref{Minardi}. 

\begin{prop}\label{thm1}
Let $p$ be a prime number. 
Then, 
there exist infinitely many imaginary quadratic fields $k$ such that $p$ splits in $k$ 
and that $X_{\tilde{k}}\sim 0$. 
\end{prop}

By the way, 
in general, 
to understand whether $X_{\tilde{k}}=0$ or not is difficult. 
For each prime number $p$, 
we can find an infinite family of imaginary quadratic fields $k$ in which $p$ splits 
such that $X_{\tilde{k}}=0$. 
Let $\lambda_p(k)$ be the Iwasawa $\lambda$-invariant of the cyclotomic 
$\Z_p$-extension of $k$. 
When $p=2$, 
let $q$ be a prime number with $q\equiv 7\bmod{16}$ and let $k=\Q(\sqrt{-q})$. 
As explained in the above, 
the prime number $2$ splits in $k$, 
and it holds that $2\nmid h_k$. 
By the result proved by Ferrero \cite{Ferrero} and by Kida \cite{kida} independently, 
it holds that $\lambda_2(k)=1$. 
For odd prime numbers $p$, 
by theorem 1.1 of Byeon \cite{Byeon}, 
there exist infinitely many imaginary quadratic fields $k$ in which $p$ splits 
such that $\lambda_p(k)=1$. 
It is known that, 
if a prime number $p$ splits in an imaginary quadratic field $k$ and if 
$\lambda_p(k)=1$, 
then $X_{\tilde{k}}=0$. 
Thus, 
we have the following.

\begin{prop}\label{thm2}
Let $p$ be a prime number. 
Then there exist infinitely many imaginary quadratic fields $k$ in which $p$ splits 
such that $X_{\tilde{k}}=0$.
\end{prop}

We are then interested in finding families of number fields $k$ such that 
$X_{\tilde{k}}\neq 0$ and $X_{\tilde{k}}\sim 0$. 
In this article, 
we show the following. 

\begin{thm}\label{main1}
Let $p$ be a prime number. 
Then there exist infinitely many imaginary abelian fields $k$ such that $X_{\tilde{k}}\neq 0$  
and $X_{\tilde{k}}\sim 0$.
\end{thm}

Some results on sufficient conditions of the pseudo-nullity of $X_{\tilde{k}}$ 
for abelian fields or CM-fields $k$ under elementary situations had been obtained, 
see \cite{Minardi}, \cite{Itoh} and \cite{Fujii}. 
However, 
as stated in the above, 
even if we know the pseudo-nullity of $X_{\tilde{k}}$, 
further to know whether $X_{\tilde{k}}=0$ or not is basically difficult. 
Hence, 
the author thinks that what the existence of families of imaginary abelian fields of theorem 
\ref{main1} are guaranteed is important in the study of GGC. 
As we will see in the proof of theorem \ref{main1}, 
we choose $k$ as imaginary quadratic fields if $p=2$, 
and as imaginary cyclic fields of degree $2p$ if $p>2$. 
For each odd prime number $p$, 
to find an infinite family of imaginary quadratic fields $k$ such that 
$X_{\tilde{k}}\neq 0$ and $X_{\tilde{k}}\sim 0$ is an interesting problem. 

In section $3$, 
we will give the proof of theorem \ref{main1}. 
In section $4$, 
we will give an application to non-abelian Iwasawa theory in the sense of Ozaki \cite{Ozaki}. 

\section{Proof of theorem \ref{main1}}
In this section, 
we prove theorem \ref{main1}. 
For a number field $K$ and a prime number $p$, 
let $K_{\infty}/K$ be the cyclotomic $\Z_p$-extension 
and $\lambda_p(K)$ the Iwasawa $\lambda$-invariant of $K_{\infty}/K$. 
We need the following well known lemma. 

\begin{lem}\label{some}
Let $p$ be a prime number and $K/\Q$ a finite extension. 
Suppose that $p$ splits completely in $K/\Q$. 
Then $\tilde{K}/K_{\infty}$ is an unramified abelian pro-$p$ extension. 
\end{lem}

\begin{proof}
Let $\f{p}$ be a prime of $K$ lying above $p$ and 
$I$ the inertia subgroup of $\f{p}$ in $\tilde{K}/K$. 
Then $\f{p}$ is unramified in $K/\Q$ and has degree $1$. 
This implies that $I\simeq \Z_p$. 
Since $K_{\infty}=K\Q_{\infty}$ and $\Q_{\infty}/\Q$ is totally ramified at $p$, 
$K_{\infty}/K$ is ramified at $\f{p}$. 
Thus the restriction map $I\to \ga(K_{\infty}/K)$ is injective, 
and hence $\tilde{K}/K_{\infty}$ is unramified at all primes lying above $\f{p}$. 
Since $\tilde{K}/K$ is unramified outside primes lying above $p$, 
we can conclude that $\tilde{K}/K_{\infty}$ is unramified at all primes. 
\end{proof} 

First, 
suppose that $p=2$. 
Let $q$ be a prime number such that $q\equiv 15 \bmod{16}$ and 
put $k=\Q(\sqrt{-q})$. 
Then, 
the prime $2$ splits in $k$, 
and by genus theory, 
we have $2\nmid h_k$. 
Thus $X_{\tilde{k}}$ is a pseudo-null $\Lambda$-module by proposition \ref{Minardi}.
Furthermore, 
by the result of Ferrero \cite{Ferrero} and Kida \cite{kida}, 
it holds that $\lambda_2(k)\geq 3$. 
Since $\tilde{k}/k_{\infty}$ is an unramified $\Z_2$-extension by lemma \ref{some}, 
there is a surjective morphism $X_{\tilde{k}} \to \Z_2^2$. 
Therefore it holds that $X_{\tilde{k}}\neq 0$. 

Next, 
suppose that $p\geq 3$. 
We show the following. 

\begin{thm}\label{mainthm2}
Let $p$ be an odd prime number. 
Then there exist infinitely many imaginary cyclic fields $k$ of degree $2p$ 
in which $p$ splits completely such that $X_{\tilde{k}}\neq 0$ and $X_{\tilde{k}}\sim 0$. 
\end{thm}

In principle, 
for each odd prime number $p$, 
we can find an imaginary abelian field $k$ which satisfies the conditions of theorem \ref{mainthm2}. 

From here, 
we begin the proof of theorem \ref{mainthm2}. 
Let $p$ be an odd prime number. 
We need the following. 

\begin{prop}[Theorem $1$ of \cite{Fujii}]\label{prop1}
Let $p$ be an odd prime number, 
$k$ a CM-field of degree greater than $2$ and $k^+$ the maximal totally real subfield of $k$. 
Suppose that, 
the prime number $p$ splits completely in $k/\Q$, 
$p\nmid h_k$, 
and all of the Iwasawa $\lambda$-, 
$\mu$- and $\nu$-invariants of the cyclotomic $\Z_p$-extension $k_{\infty}^+$ of 
$k^+$ are $0$. 
Then $X_{\tilde{k}}\sim 0$. 
\end{prop}

{\bf Remarks.} 
$(1)$ In \cite{Itoh}, 
Itoh showed a result for quartic imaginary abelian fields precisely analogous to proposition 
\ref{Minardi}. 
Proposition \ref{prop1} is a generalization of Itoh's result.
\\
$(2)$ 
In theorem $1$ of \cite{Fujii}, 
the author put the assumption that Leopoldt's conjecture holds for 
$p$ and $k^+$. 
The author must remark here that, 
if an odd prime number $p$ splits completely in $k^+/\Q$, 
Leopoldt's conjecture for $p$ and $k^+$ is implied by the vanishing of 
Iwasawa invariants $\lambda$ and $\mu$ of $k_{\infty}^+/k^+$, 
that is, GC for $p$ and $k$ holds. 
For this, 
see proposition $1$ of \cite{ozaki1997}.  
Hence, 
in theorem $1$ of \cite{Fujii}, 
the assumption on Leopoldt's conjecture is not needed. 
\\

By proposition \ref{prop1}, 
to prove theorem \ref{mainthm2}, 
it suffices to find infinitely many imaginary abelian fields $k$ of degree $2p$ 
which satisfy the conditions of proposition \ref{prop1} and $X_{\tilde{k}}\neq 0$.

When $p=3$, 
put $F=\Q(\sqrt{-47})$. 
It is known that $h_{F}=5$ (the author checked by using Pari/gp \cite{PARI2}). 
When $p\geq 5$, 
By Horie-\^{O}hnishi's result \cite{Horie-Ohnishi}, 
there is an imaginary quadratic field $F$ such that the prime $p$ splits in $F/\Q$ and 
that $p\nmid h_{F}$. 
Alternatively, 
Ito \cite{Ito} showed that the class numbers of imaginary quadratic fields 
$\Q(\sqrt{1-p})$ and $\Q(\sqrt{4-p})$ 
are not divisible by $p$, 
see lemma $2.4$ of \cite{Ito}, 
hence we can choose $F$ as $\Q(\sqrt{1-p})$ or $\Q(\sqrt{4-p})$. 
We fix once such an imaginary quadratic field $F$ for each odd prime number $p$. 

For a positive integer $n$, 
let $\mu_n$ be the group of $n$-th roots of unity. 
For each non-negative integer $n$ and a number field $K$, 
let $K_n$ be the $n$-th layer of the cyclotomic $\Z_p$-extension $K_{\infty}/K$. 

\begin{lem}\label{lem2}
Restrictions of Galois groups induce the following isomorphism
$$
\ga(F\Q_1\Q(\mu_p,\sqrt[p]{p})/\Q)
\simeq 
\ga(F/\Q)
\times
\ga(\Q_1/\Q)
\times
\ga(\Q(\mu_p,\sqrt[p]{p})/\Q),
$$
where $\Q_1$ denotes the $1$-st layer of the cyclotomic $\Z_p$-extension of $\Q$. 
\end{lem}

\begin{proof}
Since $[F:\Q]=2$ and $[\Q_1:\Q]=p>2$, 
it holds that 
$F\cap \Q_1=\Q$. 
Note that $\Q(\mu_p)/\Q$ is abelian and 
$\Q(\mu_p,\sqrt[p]{p})/\Q$ is non-abelian. 
This implies that 
$$
(F\Q_1)\cap \Q(\mu_p,\sqrt[p]{p}) 
=
(F\Q_1)\cap \Q(\mu_p).
$$
Since $[\Q(\mu_p):\Q]=p-1$, 
it holds that $\Q_1\not\subseteq (F\Q_1)\cap \Q(\mu_p)$. 
If $F\subseteq (F\Q_1)\cap \Q(\mu_p)$ then $F$ is a subfield 
of $\Q(\mu_p)$ which is unramified at $p$. 
This contradicts to the fact that $\Q(\mu_p)/\Q$ is totally ramified at $p$. 
Hence $F\not\subseteq (F\Q_1)\cap \Q(\mu_p)$. 
Therefore, 
it holds that
$$
(F\Q_1)\cap \Q(\mu_p,\sqrt[p]{p})=(F\Q_1)\cap \Q(\mu_p)=\Q.
$$
This completes the proof. 
\end{proof}

By the Chebotarev density theorem and lemma \ref{lem2}, 
there exist infinitely many prime numbers $\ell$ such that all of the following 
three conditions are satisfied: 

\begin{itemize}
\item[$(1)$]
$\ell$ is inert in $F/\Q$, 
\item[$(2)$]
$\ell$ is inert in $\Q_1/\Q$, 
\item[$(3)$]
$\ell$ splits completely in $\Q(\mu_p,\sqrt[p]{p})/\Q$. 
\end{itemize}

We fix once such a prime number $\ell$. 
For a finite prime $\f{q}$ of a number field, 
we identify finite primes and prime numbers when a number field is $\Q$, 
let $\F_{\f{q}}$ be the residue class field at $\f{q}$. 

\begin{lem}\label{lem3}
A prime number $\ell$ splits completely in $\Q(\mu_p,\sqrt[p]{p})/\Q$ 
if and only if the following two conditions are satisfied. 
\begin{itemize}
\item
$\ell \equiv 1\bmod{p}$.
\item
$p\bmod{\ell}\in (\F_{\ell}^{\times})^p$.  
\end{itemize}
\end{lem}

\begin{proof}
Let $\f{L}$ be a prime of $\Q(\mu_p)$ lying above $\ell$. 
The prime number $\ell$ splits completely in $\Q(\mu_p)$ 
if and only if $\ell\equiv 1\bmod{p}$. 
The prime $\f{L}$ splits completely in $\Q(\mu_p,\sqrt[p]{p})$ 
if and only if the equation $X^p-p \equiv 0 \bmod{\f{L}}$ has a root in $\F_{\f{L}}$ 
since $\Q(\mu_p,\sqrt[p]{p})/\Q(\mu_p)$ is unramified at $\f{L}$. 
This assertion is equivalent to that $p\bmod{\f{L}}\in (\F_{\f{L}}^{\times})^p$. 
Since $\f{L}$ has degree $1$, 
we have $\F_{\ell}\simeq \F_{\f{L}}$. 
Hence $p\bmod{\f{L}}\in (\F_{\f{L}}^{\times})^p$ if and only if 
$p\bmod{\ell}\in (\F_{\ell}^{\times})^p$. 
\end{proof}

By the condition $(3)$ and lemma \ref{lem3}, 
there is the unique subfield $k^+$ of $\Q(\mu_{\ell})$ such that $[k^+:\Q]=p$ 
in which $p$ splits completely. 
Since $k^+/\Q$ is unramified outside $\ell$, 
it holds that $p\nmid h_{k^+}$ by lemma \ref{Iwasawa56}. 

It is well known that $p\nmid h_{\Q_1}$ by lemma \ref{Iwasawa56}. 
By the condition $(2)$, 
the prime $\ell$ is inert in $\Q_1/\Q$. 
Since $k_1^+/\Q_1$ is ramified at only the unique prime of $\Q_1$ lying above $\ell$, 
it holds that $p\nmid h_{k_1^+}$ by lemma \ref{Iwasawa56}. 
Thus we have checked that $p$ does not divide both of $h_{k^+}=h_{k_0^+}$ and 
$h_{k_1^+}$. 
Then, 
by Fukuda's result \cite{Fukuda}, 
we have $p\nmid h_{k_n^+}$ for all non negative integers $n$. 
This implies that all of Iwasawa invariants $\lambda$, 
$\mu$ and $\nu$ of $k_{\infty}^+/k^+$ are $0$. 

Put $k=Fk^+$. 
From the choices of $F$ and $k^+$, 
$p$ splits completely in $k/\Q$. 
From the condition $(1)$, 
the prime $\ell$ is inert in $F/\Q$. 
Since $k/F$ is ramified at only the unique prime of $F$ lying above $\ell$ 
and since $p\nmid h_{F}$, 
it holds that $p\nmid h_k$ by lemma \ref{Iwasawa56}. 
By combining all of the above arguments, 
we have seen that the imaginary cyclic field $k$ of degree $2p$ satisfies 
all of the assumptions of proposition \ref{prop1}. 

Finally, 
we show that $X_{\tilde{k}}\neq 0$. 
Since $p$ splits completely in $k/\Q$, 
Leopoldt's conjecture holds for $k$ and $p$, 
and since $[k:\Q]=2p$, 
by lemma \ref{some}, 
it holds that $\tilde{k}/k_{\infty}$ is an unramified $\Z_p^p$-extension. 
Suppose that $p=3$. 
By using Mizusawa's software \cite{Mizusawa}, 
we find that $\lambda_3(F)=\lambda_3(\Q(\sqrt{-47}))=2$. 
By Kida's formula \cite{Kida}, 
one sees that $\lambda_3(k)=6$. 
As stated in the above, 
it holds that $\tilde{k}/k_{\infty}$ is an unramified $\Z_3^3$-extension, 
and therefore, 
there is a surjective morphism $X_{\tilde{k}}\to \Z_3^3$. 
In particular, 
we have $X_{\tilde{k}}\neq 0$. 

Suppose that $p>3$. 
For an algebraic extension $K/\Q$ and a prime number $p$, 
let $X_K$ be the Galois group of the maximal unramified abelian pro-$p$ extension of $K$. 
When $K/\Q$ is finite, 
let $\gamma$ be a topological generator of $\ga(K_{\infty}/K)$, 
and put $R=\Z_p[\![\ga(K_{\infty}/K)]\!]$. 
We need the following two lemmas. 

\begin{lem}\label{rank} 
Let $p$ be a prime number and let $K/\Q$ be a finite extension. 
Let $r_1$ and $r_2$ be the number of real primes and the number of complex primes 
of $K$ respectively. 
For each non negative integer $n$, 
let $E_n$ be the unit group of $K_n$. 
Let $E=\varprojlim_{n}E_n\otimes\Z_p$ be the projective limit of modules 
$E_n\otimes\Z_p$ with respect to the norm maps. 
Suppose that $K_{\infty}/K$ is totally ramified at all primes of $K$ lying above $p$. 
Then the $\Z_p$-rank of $E/(\gamma-1)E$ is $r_1+r_2$. 
\end{lem}

{\bf Remark.} 
Lemma \ref{rank} can be deduced from theorem 10.3.25 and theorem 11.3.11 
of \cite{NSW}. 
Because we need only to know the $\Z_p$-rank of $E/(\gamma-1)E$ here, 
we prefer to prove lemma \ref{rank} briefly. 

\begin{proof}
For a $R$-module $M$, 
let $M^{\ga(K_{\infty}/K)}$ be the maximal submodule of $M$ on which 
$\ga(K_{\infty}/K)$ acts trivially. 
For each non negative integer $n$, 
let $U_{\f{p}_n}$ be the principal local unit group at a prime $\f{p}_n$ of $K_n$ lying 
above $p$, 
and let $U=\varprojlim_{n}\oplus_{\f{p}_n\mid p}U_{\f{p}_n}$ be the projective limit of modules 
$\oplus_{\f{p}_n\mid p}U_{\f{p}_n}$ with respect to the norm maps. 
By theorem 25 of \cite{Iwasawa73}, 
we can see that $U$ is a finitely generated $R$-module of rank $[K:\Q]=r_1+2r_2$, 
and $U^{\ga(K_{\infty}/K)}=0$. 
Let $\f{X}$ be the Galois group of the maximal abelian pro-$p$ 
extension over $K_{\infty}$ unramified outside all primes lying above $p$. 
By class field theory and by remark 2 of theorem 4.2 of \cite{Kuz'min}, 
there is the following exact sequence
$$
0 \to E\to U\to \f{X}\to X_{K_{\infty}}\to 0
$$
of $R$-modules. 
It is well known that the $R$-rank of $\f{X}$ is $r_2$, 
see theorem 17 of \cite{Iwasawa73} 
(see also theorem 13.31 of \cite{Washington}). 
Combining the above, 
since $X_{K_{\infty}}$ is a torsion $R$-module, 
we find that $E$ is a finitely generated $R$-module of rank $r_1+r_2$, 
and that $E^{\ga(K_{\infty}/K)}=0$. 
In particular, 
$E$ has no non trivial finite $R$-submodules. 

Let $T_RE$ be the $R$-submodule of $E$ which consists of all $R$-torsion elements of $E$. 
By the structure theorem of $R$-modules, 
we have an injective morphism $E \to R^{r_1+r_2}\oplus T_RE$ with a finite cokernel. 
Then we have an exact sequence
$$
\mbox{(finite)} \to E/(\gamma-1)E
\to 
\Z_p^{r_1+r_2}\oplus T_RE/(\gamma-1)T_RE
\to
\mbox{(finite)}
$$
of $\Z_p$-modules. 
Since $E^{\ga(K_{\infty}/K)}=0$, 
we find that $T_RE/(\gamma-1)T_RE$ is finite. 
Therefore the $\Z_p$-rank of $E/(\gamma-1)E$ is $r_1+r_2$. 
\end{proof}

\begin{lem}\label{p=3}
Let $p$ be a prime number and $K/\Q$ a totally imaginary 
finite extension in which the prime number 
$p$ splits completely. 
Assume that Leopoldt's conjecture for $p$ and $K$ holds. 
Then we have $[K:\Q]\leq 6$ if $X_{\tilde{K}}=0$.  
\end{lem}

\begin{proof}
For an algebraic extension $L/\Q$, 
not necessary finite, 
let $\ca{G}_{L}$ be the Galois group of the maximal unramified pro-$p$ 
extension of $L$. 
It holds that $[K:\Q]=2r_2$. 
By lemma \ref{some}, 
$\ca{G}_{\tilde{K}}$ is a closed normal subgroup of $\ca{G}_{K_{\infty}}$, 
and it holds that 
$$
\ca{G}_{K_{\infty}}/\ca{G}_{\tilde{K}}\simeq \ga(\tilde{K}/K_{\infty}).
$$ 
Assume that $X_{\tilde{K}}=0$. 
By pro-$p$ version of Burnside's basis theorem, 
if $X_{\tilde{K}}=0$ then $\ca{G}_{\tilde{K}}=1$. 
By lemma \ref{some}, 
it holds that 
$$
\ca{G}_{K_{\infty}}=X_{K_{\infty}}=\ga(\tilde{K}/K_{\infty})\simeq \Z_p^{r_2},
$$
since Leopoldt's conjecture for $p$ and $K$ holds. 
Note that $\ga(K_{\infty}/K)$ acts on $X_{K_{\infty}}$ trivially since $\tilde{K}/K$ is abelian. 

By theorem C of \cite{Ozaki} for $i=2$, 
since $\ca{G}_{K_{\infty}}=X_{K_{\infty}}$, 
we have a surjective morphism
$$
E \to H_2(X_{K_{\infty}},\Z_p)\simeq X_{K_{\infty}}\wedge X_{K_{\infty}},
$$
of $R$-modules, 
where the exterior product $\wedge$ is taken over $\Z_p$. 
The action of $\gamma$ on $x\wedge y\in X_{K_{\infty}}\wedge X_{K_{\infty}}$ 
is given by 
$$
\gamma(x\wedge y)=(\gamma x)\wedge (\gamma y).
$$ 
Since $\ga(K_{\infty}/K)$ acts on $X_{k_{\infty}}$ trivially, 
the above morphism factors through $E/(\gamma-1)E$.  
Thus, 
since 
$$
X_{K_{\infty}}\wedge X_{K_{\infty}}\simeq \Z_p^{\frac{r_2(r_2-1)}{2}},
$$ 
we also have a surjective morphism
$$
E/(\gamma-1)E\to \Z_p^{\frac{r_2(r_2-1)}{2}}.
$$
By the assumption that $K$ is totally imaginary and by lemma \ref{rank}, 
the $\Z_p$-rank of $E/(\gamma-1)E$ is equal to $r_2$. 
This implies that 
$$
\frac{r_2(r_2-1)}{2}\leq r_2.
$$
Therefore we have $[K:\Q]=2r_2\leq 6$. 
\end{proof}

We finish the proof of theorem \ref{mainthm2}. 
By our assumption that $p>3$, 
it holds that $[k:\Q]=2p>6$. 
By lemma \ref{p=3}, 
we have $X_{\tilde{k}}\neq 0$. 
Since there are infinitely many prime numbers such as $\ell$, 
this completes the proof of theorem \ref{mainthm2}. 
\qed
\\

{\bf Remark.}
Okano \cite{Okano} had obtained a result stricter than lemma \ref{p=3} 
for certain imaginary abelian fields. 
In fact, 
the non-triviality of $X_{\tilde{k}}$ is also deduced from theorem 1.2 of \cite{Okano} 
when $p>3$. 

\section{Non-freeness conjecture}
We give an application to non-abelian Iwasawa theory in the sense of Ozaki \cite{Ozaki}. 
Let $p$ be a prime number and $k_{\infty}/k$ the cyclotomic $\Z_p$-extension of 
a number field $k$. 
Recall we had denoted by $\ca{G}_{k_{\infty}}$ the Galois group of 
the maximal unramified pro-$p$ extension 
over $k_{\infty}$ in the proof of lemma \ref{p=3}. 
In his lecture \cite{OzakiConj}, 
Ozaki proposed the following conjecture.

\begin{conj1}[Non-freeness Conjecture \cite{OzakiConj}]
For each prime number $p$ and each finite extension $k/\Q$, 
the group $\ca{G}_{k_{\infty}}$ never be a non-abelian free pro-$p$ group. 
\end{conj1} 

Each abelian free pro-$p$ group is isomorphic to $\Z_p$. 
We have checked that there exist infinitely many imaginary quadratic fields 
$k$ in which $p$ splits 
such that $X_{k_{\infty}}\simeq \Z_p$ for each prime number $p$, 
and hence 
$
\ca{G}_{k_{\infty}}
\simeq \Z_p
$, 
see proposition \ref{thm2}. 

In \cite{fujii}, 
the author showed following. 

\begin{lem}[\cite{fujii}]\label{F2011}
Let $p$ be a prime number and $k/\Q$ a finite extension. 
If a prime number $p$ splits completely in $k$ and $X_{\tilde{k}}\sim 0$, 
then $\ca{G}_{k_{\infty}}$ is not a non-abelian free pro-$p$ group. 
\end{lem}

As a consequence of lemma \ref{F2011}, the proof of theorem \ref{main1} 
and results given by several authors, 
we have the following. 

\begin{cor}\label{main2}
Let $p$ be a prime number. 
Then there exist infinitely many imaginary abelian fields $k$ satisfying the 
following three conditions.
\\
$(1)$ 
The prime number $p$ splits completely in $k/\Q$. 
\\
$(2)$ 
$\ca{G}_{k_{\infty}}$ is a non-abelian pro-$p$ group.
\\
$(3)$ 
$\ca{G}_{k_{\infty}}$ is not a free pro-$p$ group. 
\end{cor}

\begin{proof}
Let $p=2$ and $q$ be a prime number with $q\equiv 31\bmod{32}$. 
Put $k=\Q(\sqrt{-q})$. 
Then, 
by proposition \ref{Minardi}, 
$X_{\tilde{k}}$ is a pseudo-null $\Lambda$-module. 
Hence, 
by lemma \ref{F2011}, 
$\ca{G}_{k_{\infty}}$ is not a non-abelian free pro-$2$ group. 
Further, 
by theorem 2 of \cite{Mizusawa-Ozaki}, 
$\ca{G}_{k_{\infty}}$ is not abelian. 

Let $p\geq 3$. 
We shall use notations the same to that are in the proof of theorem \ref{mainthm2}. 
We have checked that $X_{\tilde{k}}$ is a pseudo-null $\Lambda$-module. 
By lemma \ref{F2011}, 
$\ca{G}_{k_{\infty}}$ is not a non-abelian free pro-$p$ group. 
By theorem $1.2$ of \cite{Okano}, 
$\ca{G}_{k_{\infty}}$ is not abelian, 
since $\lambda_3(F)=2$ when $p=3$. 
\end{proof}

\section*{Acknowledgments}
The author would like to express his thanks to Takuya Yamauchi, 
the study of this article was motivated by questions from Yamauchi to the author. 
This research was partly supported by JSPS KAKENHI Grant Numbers 
JP18K03259, JP19H01783.

\end{document}